\documentclass{amsart}

\newtheorem{theorem}{Theorem}[section]
\newtheorem{corollary}[theorem]{Corollary}
\newtheorem{lemma}[theorem]{Lemma}
\newtheorem{proposition}[theorem]{Proposition}

\theoremstyle{definition}
\newtheorem{definition}{Definition}[section]
\newtheorem{remark}{Remark}[section]

\def\R{{\mathbb R}}
\def\N{{\mathbb N}}

\def\qqfa{\quad{\rm for\ all}\quad}

\def\<{\langle}
\def\>{\rangle}

\def\dist{{\rm dist}}

\def\AA{{\mathbb A}}

\def\qq#1{\qquad\mbox{#1}\qquad}

\def\be#1{\begin{equation}\label{#1}}
\def\ee{\end{equation}}

\usepackage{nicefrac, mathrsfs}

\newcommand{\eval}[2][\right]{\relax
  \ifx#1\right\relax \left.\fi#2#1\rvert}

\title[Finite-dimensional attractors]{On finite-dimensional attractors of homeomorphisms}
\author{James C.\ Robinson}
\address{Mathematics Institute, University of Warwick, Coventry, CV4 7AL. UK.}
\email{j.c.robinson@warwick.ac.uk}
\author{Jaime J. S\'{a}nchez-Gabites}
\address{Departamento de Econom\'{i}a cuantitativa. Facultad de Ciencias Econ\'{o}micas y Empresariales. Universidad Aut\'{o}noma de Madrid, 28049 Cantoblanco (Madrid), SPAIN}
\email{JaimeJ.Sanchez@uam.es}
\thanks{JCR is currently an EPSRC Leadership Fellow, grant \# EP/G007470/1. JJSG is supported by a MICINN grant (MTM 2009-07030).}

\begin{document}

\begin{abstract}
Let $E$ be a linear space and suppose that $\AA$ is the global attractor of either (i) a homeomorphism $F:E\rightarrow E$ or (ii) a semigroup $S(\cdot)$ on $E$ that is injective on $\AA$. In both cases $\AA$ has trivial shape, and the dynamics on $\AA$ can be described by a homeomorphism $F:\AA\rightarrow \AA$ (in the second case we set $F=S(t)$ for some $t>0$). If the topological dimension of $\AA$ is finite we show that for any $\epsilon>0$ there is an embedding $e:\AA\rightarrow\R^k$, with $k\sim{\rm dim}(A)$, and a (dynamical) homeomorphism $f:\R^k\rightarrow\R^k$ such that $F$ is conjugate to $f$ on $\AA$ (i.e.\ $F|_\AA=e^{-1}\circ f\circ e$) and $f$ has an attractor $A_f$ with $e(\AA)\subset A_f\subset N(e(\AA),\epsilon)$. In other words, we show that the dynamics on $\AA$ is essentially finite-dimensional.

We characterise subsets of $\R^n$ that can be the attractors of homeomorphisms as cellular sets, give elementary proofs of various topological results connected to Borsuk's theory of shape and cellularity in Euclidean spaces, and prove a controlled homeomorphism extension theorem.
\end{abstract}

\maketitle
\tableofcontents


\section{Introduction}

One can recast many of the important equations of mathematical physics within the framework of infinite-dimensional dynamical systems, i.e.\ dynamical systems evolving in an infinite-dimensional phase space. The theory of such systems has been systematically developed over the last three decades, and is well covered in the monographs by Babiin \& Vishik (1992), Chepyzhov \& Vishik (2002), Chueshov (2002), Hale (1988), Ladyzhenskaya (1991), Robinson (2001), and Temam (1988). One of the most striking results in this theory is that in many interesting examples the long-time dynamics can be captured by a finite-dimensional  subset of the ambient (infinite-dimensional) phase space, the `global attractor'.

However, this statement says nothing a priori about the dynamics restricted to the attractor, and it is natural to ask in what sense (if any) these dynamics are themselves finite-dimensional. This question was first posed in this generality by Eden et al.\ (1994), and subsequently discussed by Robinson (1999) and Romanov (2000).
Ideally one would construct a finite-dimensional ordinary differential equation whose dynamics reproduces those on the attractor. This is certainly possible if the original system possesses an inertial manifold (Foias et al., 1988), but the existence of such an object requires restrictive conditions (a `spectral gap condition' on the linear part of the equation) that prevent the theory being applicable to many important examples, such as the two-dimensional Navier--Stokes equations.

 However, the construction of such an ODE seems very difficult, essentially because it would require a bi-Lipschitz embedding of the attractor into a Euclidean space (this can be slightly weakened to allow logarithmic corrections, see Pinto de Moura et al., 2011). An intrinsic characterisation of sets that admit such an embedding is a major open problem in the theory of metric spaces (see Heinonen, 2003, for example), and there are examples due to Eden et al. (2011) that show that such an embedding (even with a logarithmic correction) is in general not possible for the attractors of infinite-dimensional dynamical systems.

 Thus in this paper we aim to reproduce not the continuous dynamics on the attractor, but the discrete dynamics that come from considering the time $T$ map of the flow, for some fixed $T>0$. Perhaps a little more elegantly, we also consider the discrete problem from the outset, where the attractor arises from the iteration of some given homoeomorphism.

At the heart of our construction is a classical theorem due to Menger (1926) and N\"obeling (1931) (see also Hurewicz \& Wallman, 1941; Robinson, 2011), which guarantees that any finite-dimensional compact metric space can be embedded into a finite-dimensional Euclidean space of comparable dimension. The following theorem states this more precisely.

\begin{theorem}[Menger--N\"obeling]\label{MN}
Let $(X,d)$ be a compact metric space of dimension $\le d$. Then the set of all homeomorphisms of $X$ onto a subset of $\R^{2d+1}$ is dense in $C^0(X,\R^{2d+1})$. [In fact the homeomorphisms form a dense $G_\delta$ in $C^0(X,\R^{2d+1})$.]
\end{theorem}

We use this theorem to make a homeomorphic copy $X$ of our original finite-dimensional attractor, along with its dynamics, into some $\R^k$. The main task is to find a way to extend the embedded dynamics from $X$ onto the whole of $\R^k$, and to make the set $X$ an attractor (or as nearly as possible) for the resulting dynamics.

In Section 2 we show that cellularity characterises the global attractors of homemorphisms in Euclidean spaces, based in part on previous work of Garay (1991). In Section 3 we show that global attractors of homeomorphisms, and of continuous time semigroups, have trivial shape, a property that is topologically invariant. We recall in Section 4 the cellularity criterion of McMillan (1964), which allows us to boost a set with trivial shape to one that is cellular by adding an extra dimension to the ambient space. In Section 5 we prove a theorem that provides a controlled extension of a homeomorphism from a compact subset of $\R^k$ to a map on $\R^{2k}$, based on a trick of Klee (1955). Finally we combine these techniques in Section 6 to show that discrete dynamics on finite-dimensional attractors are no more complicated than the dynamics on attractors of homeomorphisms in finite-dimensional spaces.

\section{Attractors and cellularity}

\subsection{Attractors in linear spaces}

Suppose that $E$ is a linear space, and $F:E\to E$ is a continuous map. Then $\AA$ is a \emph{global attractor} for $F$ if
\begin{enumerate}
\item $\AA$ is compact;
\item $F(\AA)=\AA$; and
\item $\AA$ attracts bounded sets, i.e.\ for every bounded subset $B$ of $E$,
$$
\dist(F^n(B),\AA)\to0\qq{as} n\to\infty,
$$
where $\dist(A,B)=\sup_{a\in A}\inf_{b\in B}\|a-b\|$.
\end{enumerate}
Note that if it exists the global attractor is unique and is the minimal closed set that attracts bounded sets (both follow from the fact that if $Z$ attracts bounded sets then $\dist(\AA,Z)=\dist(F^n(\AA),Z)\to0$ as $n\to\infty$).

Condition (3) is equivalent to requiring that for any bounded set $B$ of $E$ and every $\epsilon > 0$ there exists $n_0$ such that $F^n(B) \subset N(X,\epsilon)$ whenever $n \geq n_0$. Here $N(X,\epsilon)$ denotes the set $\{x \in E :\ \dist(x,\AA)<\epsilon\}$. In particular, this implies that for any bounded set $B$ containing $\AA$, the equality $\AA = \bigcap_n F^n(B)$ holds.

The existence of a global attractor is equivalent to the existence of a compact attracting set $K$ (cf.\ the corresponding results for semiflows given by Crauel, 2001).

\begin{theorem}
  The map $F$ has a global attractor $\AA$ if and only if it has a compact attracting set $K$, and in this case
\be{intersection}
\AA=\bigcap_{j=1}^\infty F^j(K).
\ee
\end{theorem}

\begin{proof}
%
%
%

It is relatively straightforward to show that given the existence of a compact attracting set, for any set $B$ the set
\begin{align}
\omega(B)&:=\bigcap_{k\ge0}\overline{\bigcup_{n\ge k}F^n(B)}\label{omega1}\\
&=\{x\in E:\ x=\lim_{j\to\infty}F^{n_j}(b_j),\ n_j\to\infty,\ b_j\in B\}\label{omega2}
\end{align}
is a subset of $K$ that is compact, invariant, and attracts $B$ (see Hale, 1988, or Theorem 11.3 in Robinson, 2011, for example). We now show that $\omega(K)$ is the global attractor. Consider
\be{Xdef}
\AA=\overline{\bigcup_{B\ \mbox{bounded}}\omega(B)}.
\ee
Since $\omega(B)\subset K$ for every $B$, this is a closed subset of $K$, and so compact. It is clearly invariant since every $\omega(B)$ is invariant, and it attracts every bounded set, so it must be the global attractor since this is unique. It is immediate from (\ref{Xdef}) that $\omega(K)\subseteq\AA$. Since $\AA$ is the minimal closed set that attracts bounded sets, $\AA\subseteq K$ and hence $\AA=\omega(\AA)\subset\omega(K)$.

Finally, to show that (\ref{intersection}) holds, note that $\omega(K)\subseteq K$ and then, since $\omega(K)$ is invariant, it follows from (\ref{omega1}) that
$$
\omega(K)\subseteq\bigcap_{k\ge0}F^k(K),
$$
and
$$
F^k(K)\subset\overline{\bigcup_{n\ge k}F^n(K)}\qquad\Rightarrow\qquad\bigcap_{k=1}^\infty F^k(K)\subseteq\omega(K),
$$
which yields (\ref{intersection}).
\end{proof}

We now show that any global attractor of a \emph{homeomorphism} must be strongly cellular. Recall that a set $X$ is \emph{cellular} in $E$ if
$$
X=\bigcap_{j=1}^\infty C_j,
$$
where the $C_j$ are a decreasing sequence ($C_{j+1}\subset{\rm int}(C_j)$) of {\it cells}, i.e.\ sets homeomorphic to the closed unit ball in $E$ (`topological balls'). $X$ is \emph{strongly cellular} if for any open set $U$ containing $X$ there is an $j$ such that $C_j\subset U$. Cellularity and strong cellularity are equivalent in finite-dimensional spaces because of their local compactness [see Garay (1991) and McCoy (1973)].

%
%

\begin{lemma}\label{Xiscellular}
 If $\AA$ is the global attractor of a homeomorphism $F:E\rightarrow E$, where $E$ is a linear space, then $\AA$ is strongly cellular.
\end{lemma}

\begin{proof}
  Choose $R$ sufficiently large that $\AA\subset B(0,R)$. Since $\AA$ is the global attractor, there exists an $n$ such that
  $$
  F^n(\bar{B}(0,R))\subset B(0,R).
  $$
  It follows that $C_j = F^{nj}(\bar{B}(0,R))$ is a decreasing sequence of bounded sets all of which are cells because $F^{nj}$ is a homeomorphism for each $j$. Clearly \[\AA \subseteq \bigcap_{j=1}^\infty C_j\] because $\AA$ is invariant. Since $\AA$ is the global attractor of $F$, given any open $U\supset \AA$ there exists $j$ such that $C_j \subseteq U$. This readily implies that \[\AA = \bigcap_{j=1}^\infty C_j,\] so $\AA$ is strongly cellular.
\end{proof}

\subsection{Attractors in Euclidean spaces}

We now want to show that any cellular subset of $\R^n$ can be the global attractor of some homeomorphism on $\R^n$. We start by showing that any cellular set $X$ in $\R^n$ is `pointlike', i.e.\ $\R^n\setminus X\simeq\R^n\setminus\{0\}$.

\begin{lemma}[Brown, 1960] \label{lem:brown}
If  $X\subset\R^n$ is cellular then there exists a continuous map $g:\R^n\rightarrow\R^n$ such that $g(X)=\{0\}$, and $g|_{\R^n\setminus X}:\R^n\setminus X\rightarrow \R^n\setminus\{0\}$ is a homeomorphism. Moreover, if $X \subset B(0,R)$, then $g$ can be chosen to be the identity outside $B(0,R)$.
\end{lemma}

\begin{proof}
Let
 $$
X=\bigcap_{j=0}^\infty Q_j,
$$
where each $Q_j$ is a cell, $Q_{j+1}\subset{\rm int}(Q_j)$. Perhaps discarding the first few $Q_j$ we can assume that $Q_1$ is a proper subset of $Q := \overline{B}(0,R)$.

Let $g_1:Q\rightarrow Q$ be a homeomorphism such that $g_1|{\partial Q}={\rm Id}$ and $g_1(Q_1)\subset B(0,\nicefrac{R}{2})$. To see that such a homeomorphism exists, first note that there exists $0<r<R$ such that $Q_1$ is contained in $B(0,r)$. Consider a strictly increasing continuous map $a : [0,R]\rightarrow [0,R]$ such that $a(R) = R$, $a$ is the identity near $0$ and $a(r) = \nicefrac{R}{2}$. Then the map $g_1 : Q\rightarrow Q$ given by
$$
g_1(x)=a(|x|)\frac{x}{|x|}
$$
is a  homeomorphism (the fact that $a$ is the identity near $0$ guarantees that $h$ is continuous) such that $g_1(B(0,r))$, and hence also $g_1(Q_1)$, is contained in $B(0,\nicefrac{R}{2})$.

Now, given $g_{j-1}$, let $g_j:Q\rightarrow Q$ be a homeomorphism such that $$
g_j(x)=g_{j-1}(x)\quad x\in Q\setminus Q_{j-1}\qq{and} g_j(Q_j)\subset B(0,\nicefrac{R}{j+1}).
$$
(That such a homeomorphism exists follows from a similar argument to that given above.) Set
$$
g(x)=\lim_{j\rightarrow\infty}g_j(x);
$$
then $g:Q\rightarrow Q$ is continuous and has $g|_{\partial Q}={\rm Id}$ and $g(X)=\{0\}$ by construction. To see that $g$ is a homeomorphism of $Q\setminus X$ onto $Q\setminus\{0\}$, suppose that $x,y\in Q$ with $x\notin X$, and hence $x\notin Q_j$ for some $j$. Then either

(i) $y\in X$; then $g(y)=0$ and $g(x)\neq0$ since $g_j$ is a homeomorphism and therefore $g(x) = g_j(x) \not\in g_j(Q_j)\subset B(0,\nicefrac{R}{j+1})\ni 0$.

(ii) $y\notin X$; then for some $j$, $x,y\notin Q_j$; since $g(z)=g_j(z)$ for all $z\notin Q_j$ and $g_j:Q\rightarrow Q$ is a homeomorphism, $g(x)\neq g(y)$.

Now simply extend $g$ to $\R^n$ by letting $g$ be the identity on $\R^n\setminus Q$.
\end{proof}

Using this we can show that any cellular $X\subset\R^n$ is the attractor of some homeomorphism. (In fact one can use almost exactly the same proof to define an abstract flow on $\R^n$ that has $X$ as an attractor.)

\begin{theorem}[After Garay, 1991]\label{Garaysthm}
If $X$ is a cellular subset of $\R^n$ then there exists a homeomorphism $h:\R^n\rightarrow\R^n$ such that $h(x)=x$ for all $x\in X$ and $X$ is the global attractor for the dynamical system generated by $h$. Moreover, if $X \subset B(0,R)$, there exists a constant $\rho > 0$ such that $h(B(0,r)) \subset B(0,r-\rho)$ for every $r \geq R$.
\end{theorem}

\begin{proof} After rescaling, we may assume without loss of generality that $R > 1$. Choose $R > R' > 1$ such that $X \subset B(0,R')$ and let $g : \mathbb{R}^n \rightarrow \mathbb{R}^n$ be the map given by Lemma \ref{lem:brown}, with the property that $g$ is the identity outside the ball $B(0,R')$.
%
 Consider the annuli
  $$
  R_k=\{x\in\R^n:\ 2^{-(k+1)}\le|x|\le 2^{-k}\}
  $$
  for $k = 0,1,2,\ldots$ Each $R_k$ is a compact subset of $\R^n \backslash \{0\}$, and so $g^{-1} : \mathbb{R}^n \backslash \{0\} \rightarrow \mathbb{R}^n \backslash X$ is uniformly continuous on each $R_k$; in particular there exists a $b_k$ such that if $x,y\in R_k$ with $x=r\xi$ and $y=s\xi$, where $|\xi|=1$, then
  \begin{equation}\label{hinvsmall}
  |g^{-1}(r\xi)-g^{-1}(s\xi)|\le 2^{-k}
  \end{equation}
  provided that $|r-s|\le b_k$. Redefine $b_k$ (if necessary) to ensure that $R-b_0 > R'$ and $$b_k<\min(b_{k-1}/2,2^{-(k+3)}),$$
  and let $\beta:[0,\infty)\rightarrow[0,\infty)$ be the piecewise linear function with $\beta(0)=0$, $\beta(2^{-k})=b_k$ for each $k\in\N$, and $\beta(r)=b_0$ for $r\geq 1$. Now set $\alpha(r)=r-\beta(r)$ and note that (i) $\alpha(0)=0$, $\alpha(r)\to\infty$ as $r\to\infty$, and $\alpha:[0,\infty)\rightarrow[0,\infty)$ is strictly increasing, so $\alpha$ is a homeomorphism; (ii) $\alpha^k(r)\rightarrow0$ as $k\rightarrow\infty$ for any $r>0$; and (iii) $|r-\alpha(r)|=\beta(r)\rightarrow0$ as $r\rightarrow0$.
  For $x\notin X$ let
  $$
  h(x)=g^{-1}\bigg[\alpha(|g(x)|)\frac{g(x)}{|g(x)|}\bigg],
  $$
  and for $x\in X$ set $h(x)=x$. Clearly $X$ is the attractor of this homeomorphism, since
  $$
  g[h^k(x)]=\alpha^k(|g(x)|)\frac{g(x)}{|g(x)|},
  $$
  $\alpha^k(r)\to0$ as $k\to\infty$, and $\dist(y,X)\to0$ as $g(y)\to0$.


   This mapping satisfies the requirements of the theorem; the only possible issue is continuity at each $x\in X$. First observe the following: if $y\in g^{-1}(R_k)$ then $y=g^{-1}(r\xi)$ with $2^{-(k+1)}\le r\le 2^{-k}$ and $|\xi|=1$, so
  $$
  |y-h(y)|=|g^{-1}(r\xi)-g^{-1}(\alpha(r)\xi))|\le 2^{-k},
  $$
  since $|r-\alpha(r)|=\beta(r)\le b_k$, using (\ref{hinvsmall}) and the definition of $\beta(\cdot)$. Now fix $x \in X$ and $\epsilon > 0$. Choose $N>0$ such that $2^{-N}<\epsilon/2$ and $0<\delta<\epsilon/2$ so small that $|y-x| < \delta$ implies $y \in X$ or $y \in g^{-1}(R_k)$ for some $k \geq N$. Then if $|y-x| < \delta$
  $$
  |h(x)-h(y)|\le|x-y|+|y-h(y)|\le \epsilon/2 + 2^{-N} < \epsilon,
  $$
  which shows that $h$ is continuous.

  Finally, pick $x \in \R^n$ with $|x| = r > R$. Then $g(x) = x$ because $g$ is the identity outside $B(0,R)$, and $\beta(|x|) = b_0$ because $R$ was assumed to be bigger than $1$. Thus \[\alpha(|g(x)|)\frac{g(x)}{|g(x)|} = (|x|-b_0)\frac{x}{|x|} = (r-b_0)\frac{x}{|x|}\] has modulus $r-b_0 \geq R - b_0 > R'$, so $g^{-1}$ leaves it fixed. Consequently $h$ transforms the sphere of radius $r$ into the sphere of radius $r - b_0$, and therefore $h(B(0,r)) \subset B(0,r-b_0)$. \end{proof}

Note that the results of Lemma \ref{Xiscellular} and Theorem \ref{Garaysthm} show that cellularity is the characteristic feature of attractors of homeomorphisms in Euclidean spaces.

\section{Global attractors have trivial shape}

Given a finite-dimensional attractor $\AA$ of a homeomorphism on some linear space $E$, our aim is to construct a homeomorphism on some $\R^n$ that has a homeomorphic copy $A$ of $\AA$ as an attractor. We have shown that to be an attractor in $\R^n$ the set $A$ must be cellular, but cellularity is not a topological property, i.e\ a priori there is no way to guarantee that $A$ is cellular, even though $\AA$ is.

In order to circumvent this problem we introduce some more refined topological ideas from the theory of shape due to Borsuk (1975). Here we follow Garay (1991) and relate trivial shape to contractibility properties (the equivalence of Borsuk's original definition with that given here follows from Borsuk (1967) and Hyman (1969)).

\begin{definition} Let $A$ be a subset of $B$. A \emph{contraction of $A$ in $B$} is a continuous map \[F : A \times [0,1] \longrightarrow B\] such that $F_0 = {\rm id}_A$ and $F_1 \equiv \text{constant map}$, where $F_t$ means the partial map $F_t : A \longrightarrow B$ given by $F_t(p) := F(p,t)$. If such a contraction exists we say that $A$ is contractible in $B$. If a set $A$ is contractible in itself we simply say that $A$ is \emph{contractible}.
\end{definition}

An easy but important remark is the following:

\begin{remark}\label{contsub}
 If $A$ is contractible in $B$ and $A_0 \subseteq A$, then $A_0$ is also contractible in $B$. A suitable contraction may be obtained by restricting a contraction of $A$ in $B$.
\end{remark}

As a simple but useful example, we note that any ball $B$ in a linear space is contractible, since there is an obvious contraction onto its centre. Namely, if $B = B(x_0,r)$, then \[F(x,t) := x_0 + (x-x_0)(1-t)\] provides a contraction of $B$ in itself.

The following is an extremely trivial proposition, but nevertheless we include it here for comparison purposes with Proposition \ref{prop:trivial_top} below.

\begin{proposition}\label{contract_topo} If $h : A \longrightarrow A'$ is a homeomorphism and $A$ is contractible, then so is $A'$. That is, ``being contractible'' is a topological property.
\end{proposition}

\begin{proof} Let $F : A \times [0,1] \longrightarrow A$ be a contraction. Then \[h \circ F \circ (h^{-1} \times {\rm id}_{[0,1]}) : A' \times [0,1] \longrightarrow A'\] is a contraction of $A'$.
\end{proof}

Now we introduce the property we are interested in, which is weaker than being contractible.

\begin{definition} Let $X$ be a compact subset of a linear space $E$. We say that $X$ has \emph{trivial shape} if for every neighbourhood $U$ of $X$ in $E$, $X$ is contractible in $U$.
\end{definition}

Observe that we do not require the existence of a contraction of $X$ in itself, but that there exist contractions of $X$ that take place in arbitrarily small neighbourhoods of $X$ in $E$. We now show that cellular sets have trivial shape.

\begin{lemma}
  If $E$ is a linear space and $X\subset E$ is strongly cellular, then $X$ has trivial shape.
\end{lemma}

\begin{proof}
Since $X$ is strongly cellular, given any neighbourhood $U$ of $X$ we can find a cell $C$ such that $X\subset C\subset U$. Any ball is contractible, therefore (Proposition \ref{contract_topo}) any cell is contractible. Therefore (Remark \ref{contsub}) any subset of a cell is contractible within that cell. Thus $X$ is contractible within $C$, and so within $U$. It follows that $X$ has trivial shape.
\end{proof}

The following corollary, an immediate consequence of this result and Lemma \ref{Xiscellular}, gives one indication why this definition is potentially interesting.

\begin{corollary}\label{Xistrivial}
If $E$ is a linear space and $\AA$ is the global attractor of a homeomorphism $F:E\rightarrow E$, then $\AA$ has trivial shape.
\end{corollary}

However, unlike cellularity, having trivial shape is a topological property.

\begin{proposition} \label{prop:trivial_top} Let $h : X \longrightarrow X'$ be a homeomorphism between two compact sets $X$ and $X'$ contained in linear spaces $E$ and $E'$. Then $X$ has trivial shape if, and only if, $X'$ has trivial shape.
\end{proposition}
\begin{proof} Let $\widehat{h} : E \longrightarrow E'$ be a continuous extension of $h$ (this exists by Tietze's theorem). We assume that $X$ has trivial shape.

Let $U'$ be a neighbourhood of $X'$ in $E'$ and consider $U := \widehat{h}^{-1}U'$, which is a neighbourhood of $X$ in $E$. Since $X$ has trivial shape, it is contractible in $U$; let $F : X \times [0,1] \longrightarrow U = \widehat{h}^{-1}U'$ be a contraction. Then \[\widehat{h} \circ F \circ (h^{-1} \times {\rm id}_{[0,1]}) : X' \times [0,1] \longrightarrow U\] is a contraction of $X'$ in $U'$.
\end{proof}

One may wonder whether global attractors are actually contractible. The answer is, in general, negative. A quick way to prove this is to observe that a contractible set must be path connected, and then construct examples where global attractors exist which are not path connected. This can be done even in the plane.

The following result shows that attractors with trivial shape also arise in more general situations. A semiflow $S(\cdot):E\to E$ is a family of maps $\{S(t):\ t\ge0\}$ such that $S(0)$ is the identity map and $S(t+s)=S(t)S(s)$ for all $t,s\ge0$.

\begin{proposition} \label{prop:att_trivial2}
  If $\AA$ is the global attractor of a semiflow $S(\cdot)$ on a linear space $E$ then $\AA$ has trivial shape.
\end{proposition}

The result appears in Garay (1991), but our proof is much simpler. Note that this result is not a simple consequence of applying Corollary \ref{Xistrivial} to the map $S(T)$ for some fixed $T>0$, since there is no reason why $S(T)$ should be a homeomorphism.

\begin{proof} Let $U$ be a neighbourhood of $\AA$. We need to show that $\AA$ is contractible in $U$; that is, there exist a continuous map $F : \AA \times [0,1] \longrightarrow U$ and a point $* \in U$ such that $F(p,0) = p$ and $F(p,1) = *$ for every $p \in \AA$.

Choose any $q \in \AA$ and let $G :\AA \times [0,1] \longrightarrow E$ be defined as $G(p) := q + (1-t)(p-q)$. Clearly $G$ is a continuous map such that $G(p,0) = p$ and $G(p,1) = q$ for every $p \in \AA$. Since $\AA\times [0,1]$ is compact and $G$ is continuous, its image $C := G(\AA \times [0,1])$ is compact. Thus there exists $T > 0$ such that $S(t)(C) \subseteq U$ for every $t \geq T$, because $\AA$ attracts compact subsets of $E$.

Let $H : \AA \times [0,1] \longrightarrow U$ be the composition $H := S(T) \circ G$ (notice that the range of $H$ is now $U$). Denote $* := S(T)(q) \in \AA \subseteq U$. $H$ is clearly continuous; it also satisfies $H(p,0) = S(T)(p)$ and $H(p,1) = *$ for every $p \in \AA$. Thus it is almost a contraction of $\AA$ in $U$, the only issue being the fact that $H(p,0) = S(T)(p)$ rather than $H(p,0) = p$. However, this is easy to fix, as follows. Let $F :\AA \times [0,1] \longrightarrow U$ be defined as
$$
F(p,t) := \begin{cases} S(2Tt)(p) & \text{if } 0 \leq t \leq 1/2, \\ H(p,2t-1) & \text{if } 1/2 \leq t \leq 1. \end{cases}
$$
It is straightforward to check that $F$ is continuous and satisfies the required properties $F(p,0) = p$ and $F(p,1) = *$ for every $p \in \AA$.
\end{proof}

We note that in many interesting examples one can show that, at least on the attractor, the semigroup is injective. In this case, it follows that for any time $t>0$, the time-$t$ map $S(t)$ is a homeomorphism. What follows is therefore applicable to both the global attractors of homeomorphisms, and to the time $t$ map on any global attractor of a semigroup.

\section{Trivial shape and cellularity in Euclidean spaces}

Now suppose that we begin with a set $\AA$ that is the attractor of a homeomorphism (or a semiflow) on a linear space. We therefore know from Corollary \ref{Xistrivial} that $\AA$ has trivial shape. If $\AA$ is finite-dimensional, then we can use Theorem \ref{MN} to find an embedding $e:\AA\to\R^n$ for some $n$. Since trivial shape is a topological property (Proposition \ref{prop:trivial_top}), it follows that $e(\AA)\subset\R^n$ has trivial shape.

However, in order to make $e(\AA)$ the attractor of a homeomorphism on $\R^n$ it must be cellular. We can obtain a cellular set by appealing to the following result due essentially to McMillan (1964), but given in precisely the form we require in Daverman (1986, III.18, Corollary 5A). This gives cellularity of $X_0\times\{0\}$ in $\R^{n+1}$ whenever $X_0$ has trivial shape.


\begin{theorem}[McMillan--Daverman]\label{McMc}
If $X_0$ is a compact subset of $\R^n$ that has trivial shape then
then $X_0\times\{0\}\subset\R^{n+1}$ is cellular.
\end{theorem}

\begin{corollary}\label{embed_cell} Let $\AA$ be the global attractor of a homeomorphism or a semiflow on a linear space $E$, and assume ${\rm dim}(\AA)\le k$. Then there exists a homeomorphism $e : \AA \rightarrow X \subset \R^{2k+2}$ such that $X$ is cellular in $\R^{2k+2}$.
\end{corollary}
\begin{proof} By Theorem \ref{MN} there exists a homeomorphism $e : \AA \rightarrow X_0 \subset \R^{2k+1}$. By identifying $\R^{2k+1}$ with $\R^{2k+1} \times \{0\} \subset \R^{2k+2}$ we may think of $e$ as a homeomorphism onto the subset $X := X_0 \times \{0\} \subset \R^{2k+2}$. Proposition \ref{Xistrivial} or Proposition \ref{prop:att_trivial2} guarantee that the set $\AA$ has trivial shape, and the same is true of $X_0$ by Proposition \ref{prop:trivial_top}. Then Theorem \ref{McMc} implies that $X$ is cellular in $\R^{2k+2}$.
\end{proof}

  \section{Extension of homeomorphisms from compact subsets of $\R^n$}

  With the above results we can guarantee that the attractor $\AA\subset E$ has trivial shape and find a homeomorphism $e:\AA\to X\subset\R^{2k+2}$ such that $X$ is cellular. Note that the dynamics on $\AA$ is reproduced on $X$ by means of the homeomorphism $f=e\circ F\circ e^{-1}$ (the dynamics on $\AA$ and $X$ are conjugated by $e$).

  We now need to extend the homeomorphism $f:X\to X$ to a homeomorphism on the whole of $\R^{2k+2}$. Since we have the freedom to increase the dimension of the ambient space (we have done this once already), we can use the elegant trick to due to Klee (1955, statement (3.3)), as outlined in Proposition \ref{prop:Klee}, below.

  We make use of the following elementary results on extension of continuous functions, where $B(0,R)$ denotes the open ball of radius $R$.

  \begin{lemma}\label{Stein}
    Let $X\subset B(0,R)\subset\R^n$ be a compact set and $f:X\rightarrow X$ be a continuous function. Then there exists an extension $\varphi : \mathbb{R}^n \rightarrow \mathbb{R}^n$ of $f$ such that $\varphi$ is the identity outside $B(0,R)$.
  \end{lemma}

  \begin{proof}
    Let $D$ denote the boundary of $B(0,R)$ and extend $f$ to $X \cup D$ by letting it be the identity on $D$. Considering this as a map from $X \cup D$ into $\R^n$, the Tietze extension theorem can be used to obtain a continuous $\varphi : \overline{B}(0,R) \longrightarrow \R^n$ such that $\varphi|_X = f$ and $\varphi|_D = {\rm id}$. It only remains to set $\varphi(x) := x$ for $x \not\in B(0,R)$.
  \end{proof}

  The second simple lemma will be crucial in obtaining a controlled extension in the subsequent theorem.

    \begin{lemma}\label{compress}
  	Let $X \subset B(0,R) \subset \R^n$ be a compact set. There exists a homeomorphism $c : \mathbb{R}^n \longrightarrow \mathbb{R}^n$ such that $c|_X = {\rm id}_X$ and $c(B(0,r)) \subseteq B(0,\nicefrac{r}{2})$ for every $r \geq 2R$.
  \end{lemma}
  \begin{proof}
  	Choose $R^* < R$ be so close to $R$ that $X \subseteq B(0,R^*)$. Let $\theta : [0,+\infty) \longrightarrow [0,+\infty)$ be the (unique) continuous map that is linear on each of the intervals $[0,R^*]$, $[R^*,2R]$ and $[2R,+\infty)$ and such that (i) $\theta(r) = r$ for $0 \leq r \leq R^*$, (ii) $\theta(r) = \nicefrac{r}{2}$ for $r \geq 2R$. Notice that $\theta$ is strictly increasing and surjective, hence a homeomorphism.
  	
  	Define $c : \mathbb{R}^n \longrightarrow \mathbb{R}^n$ by \[c(p) := \left\{\begin{array}{ll} 0 & \text{ for } p = 0 \\ p \frac{\theta(|p|)}{|p|} & \text{ for } p \neq 0 \end{array} \right.\] Clearly $c$ is continuous except at, possibly, $p = 0$. However for $p \in B(0,R^*)$ we have $c(p) = p$, so $c$ is continuous at $p=0$ too. This also shows that $c|_X = {\rm id}_X$. It is easy to check that $c^{-1}$ is given by \[c^{-1}(p) = \left\{\begin{array}{ll} 0 & \text{ for } p = 0 \\ p \frac{\theta^{-1}(|p|)}{|p|} & \text{ for } p \neq 0 \end{array} \right.\] which is continuous for the same reason as above, so $c$ is a homeomorphism.
  	
  	It only remains to show that $c(B(0,r)) \subseteq B(0,\nicefrac{r}{2})$ for every $r \geq 2R$. Thus, let $r \geq 2R$ and pick $p \in B(0,r)$. If $|p| \leq 2R$, then $\theta(|p|) \leq R$ and so $c(p) \in B(0,R) \subseteq B(0,\nicefrac{r}{2})$. If $|p| > 2R$ then $\theta(p) = \nicefrac{|p|}{2} < \nicefrac{r}{2}$, so $c(p) \in B(0,\nicefrac{r}{2})$.
 	\end{proof}

\begin{proposition}\label{prop:Klee}
    Let $X\subset B(0,R)\subset\R^n$ be a compact set and $f:X\rightarrow X$ a homeomorphism. Then there exists a homeomorphism $\hat f:\R^{2n}\rightarrow\R^{2n}$ such that
    $$
    \hat f(x,0_n)=(f(x),0_n)\qqfa x\in X,
    $$
    where $0_n$ is the origin in $\R^n$, and such that
    $$
    (x,y)\in B(0,r)\times B(0,r)\quad\Rightarrow\quad \hat{f}(x,y)\in B(0,r)\times B(0,r)
    $$
    for any $r\ge R$.
  \end{proposition}

  \begin{proof}
We use Lemma \ref{Stein} to extend the continuous map $f:X\rightarrow\R^n$ to a continuous map $\varphi:\R^n\rightarrow\R^n$ which is the identity outside $B(0,R)$, and again to extend the continuous map $f^{-1}:X\rightarrow\R^n$ to another continuous map $\psi:\R^n\rightarrow\R^n$ which is the identity outside $B(0,R)$.

    Let $f_1,f_2:\R^{2n}\rightarrow\R^{2n}$ be the homeomorphisms defined by
    $$
    f_1(x,y)=(x,y+\varphi(x))\qquad\mbox{and}\qquad f_2(x,y)=(2y+\psi(x),x),
    $$
    where $x,y\in\R^n$. One can check that these are homeomorphisms since their inverses are given explicitly by
    $$
     f_1^{-1}(x,y)=(x,y-\varphi(x))\qquad\mbox{and}\qquad f_2^{-1}(x,y)=(y,\nicefrac{1}{2}(x-\psi(y))).
    $$

    Define $g=f_2^{-1}\circ  f_1$. This is a homeomorphism of $\R^{2n}$ (since it is the composition of homeomorphisms). Notice that for every $x\in X$
    $$
  f_1(x,0)=(x,\varphi(x))=(x,f(x))
    $$
    and
    $$
 f_2(f(x),0)=(\psi(f(x)),f(x)) = (x,f(x)),
    $$ so
    \begin{align*}
    g(x,0)&=(f_2^{-1}\circ f_1)(x,0)\\
    &=f_2^{-1}(x,f(x))\\
    &=(f(x),0).
    \end{align*}
    The argument to this point provides a homeomorphism that extends $f$. We now combine this with the homeomorphism of Lemma \ref{compress} to obtain the controlled extension we require.

 		Let $c : \mathbb{R}^n \longrightarrow \mathbb{R}^n$ be the homeomorphism constructed in Lemma \ref{compress}, which gives rise to another homeomorphism $\hat{c} : \mathbb{R}^{2n} \longrightarrow \mathbb{R}^{2n}$ setting $\hat{c}(x,y) := (c(x),c(y))$. Finally, define $\hat{f} := \hat{c} \circ g$. We claim that $\hat{f}$ satisfies the required properties.
 		
 		Since $\hat{c}|_{X \times \{0\}} = {\rm id}_{X \times \{0\}}$, clearly $\hat{f}(x,0) = (f(x),0)$ for $x \in X$. Now let $r \geq R$ and pick $(x,y) \in B(0,r) \times B(0,r)$. Then $|\varphi(x)| \leq r$, so
    $$
    (x',y')=f_1(x,y)=(x,y+\varphi(x))\in B(0,r)\times B(0,2r),
    $$
    and similarly $|\psi(y')| \leq 2r$, so
    $$
    g(x,y) = f_2^{-1}(x',y')=(y',\nicefrac{1}{2}(x'-\psi(y'))\in B(0,2r)\times B(0,2r),
    $$
    which shows that $g(x,y) \in B(0,2r) \times B(0,2r)$. Now $2r \geq 2R$ so $c(B(0,2r)) \subseteq B(0,r)$, and it follows that $\hat{f}(x,y) = (\hat{c} \circ g)(x,y) \in B(0,r) \times B(0,r)$.
    \end{proof}

  \section{Finite-dimensional dynamics}

  Finally we can combine these ingredients to show that the dynamics on $\AA$ occur within the attractor of a finite-dimensional system.

\begin{theorem}\label{main}
  Let $\AA$ be the attractor of a homeomorphism $F:E\to E$ with ${\rm dim}(\AA)\le k$. For any $\epsilon>0$ there exist homeomorphisms
  $$
  e: \AA\rightarrow A\subset\R^{4k+4}\qq{and} f:\R^{4k+4}\rightarrow\R^{4k+4}
  $$
  such that the dynamics on $\AA$ and $A$ are conjugate under $\varphi$, i.e.
    $$F|_\AA=e^{-1}\circ f\circ e,$$
    and $\{f^n\}$ has an attractor $A_f$ with
    $$
   A\subset A_f\subset N(A,\epsilon),
    $$
    where
    $$
    N(A,\epsilon)=\{y\in\R^{4k+4}:\ \dist(y,A)<\epsilon\}.
    $$
\end{theorem}

In other words, $A\simeq\AA$ is `almost' the attractor of $f$; and the dynamics of $F|_\AA$ are no more complicated than those of $f$.

  \begin{proof} Corollary \ref{embed_cell} provides a homeomorphism $e:\AA\to X_1\subset\R^{2k+2}$, where $X_1$ is cellular in $\R^{2k+2}$. Using Theorem \ref{Garaysthm} we may find a homeomorphism $h : \R^{2k+2} \rightarrow \R^{2k+2}$ such that $X_1$ is the global attractor of $h$ and $h|_{X_1} = {\rm Id}$. Choose $R$ such that $X_1 \subset B(0,R)$. Then we can assume there is a constant $\rho > 0$ such that $h(B(0,r)) \subset B(0,r-\rho)$ for $r \geq R$ and, in particular, $h^m(B(0,r))\subseteq B(0,\max(r-m\rho,R))$ for every $m$.


The map $f_1:X_1\to X_1$ defined by $f_1(x)=e \circ F\circ e^{-1}(x)$ is a homeomorphism. We use Proposition \ref{prop:Klee} to produce a homeomorphism $\hat{f}_1:\R^{4k+4}\to\R^{4k+4}$ such that
$$
\hat{f}_1(x,0)=(f_1(x),0)\qq{for all}x\in X_1,
$$
i.e.\ $\hat{f}_1$ extends $f_1$ from the set $A:=X_1\times\{0\}^{2k+2}$ to all of $\R^{4k+4}$. We can choose $\hat{f}_1$ so that
   $$
    (x,y)\in B(0,r)\times B(0,r)\quad\Rightarrow\quad \hat{f}_1(x,y)\in B(0,r)\times B(0,r)
    $$
    for any $r\ge R$.

Consider the homeomorphism $\hat{h} : \R^{4k+4} \rightarrow \R^{4k+4}$ given by $\hat{h}(x,y) := (h(x),y/2)$ for $x,y \in \R^{2k+2}$. Clearly $A$ is a global attractor for $\hat{h}$ and $\hat{h}|_A = {\rm Id}$. Reduce $\epsilon$, if necessary, so that $N(A,\epsilon) \subset B(0,R) \times B(0,R)$ and choose $m \geq 1$ big enough so that $\hat{h}^m(B(0,R) \times B(0,R)) \subset N(A,\epsilon)$. Set $f := \hat{h}^m \circ \hat{f}_1$. We claim that $f$ satisfies the required properties.

It is clear that $f|_A=f_1|_X$, since $\hat{h}$ is the identity on $X$. Denote $B := B(0,R) \times B(0,R)$ for brevity. Then \[f(N(A,\epsilon)) \subseteq f(B) = \hat{h}^m \hat{f}_1(B) \subseteq \hat{h}^m(B) \subseteq N(A,\epsilon),\] which shows that $N(A,\epsilon)$ is positively invariant under $f$. Consequently $f$ has an attractor $A_f$ contained in $N(A,\epsilon)$. Since $A$ is invariant, it is clear that $A \subseteq A_f$, and it only remains to show that $A_f$ is a global attractor. Since $f(B) \subset N(A,\epsilon)$, it suffices to show that for every bounded set $C$ there exists an iterate $n$ such that $f^n(C) \subset B$.

Consider the mapping $\phi(r) := \max(r-m\rho,\nicefrac{r}{2^m},R)$. We claim that \[f(B(0,r) \times B(0,r)) \subset B(0,\phi(r)) \times B(0,\phi(r)).\] This is easily seen to be true for $r \leq R$ because \[f(B(0,R) \times B(0,R)) \subseteq B(0,R) \times B(0,R),\] while for $r > R$ we have \[\hat{f}_1(B(0,r) \times B(0,r)) \subset B(0,r) \times B(0,r)\] and then \[\hat{h}^m\hat f_1(B(0,r) \times B(0,r)) \subset B(0,\max(r-m\rho,R)) \times B(0,\nicefrac{r}{2^m}) \subseteq B(0,\phi(r)) \times B(0,\phi(r)).\]

Let $C$ be a bounded subset of $\R^{4k+4}$ and choose $r$ such that $C \subset B(0,r) \times B(0,r)$. It is very easy to observe that $\phi^n(r) = R$ for $n$ big enough. Thus \[f^n(C) \subset f^n(B(0,r) \times B(0,r)) \subset B(0,\phi^n(r)) \times B(0,\phi^n(r)) = B(0,R) \times B(0,R),\] as required. \end{proof}

  \section{Conclusion and open problems}

Theorem \ref{main} shows that the dynamics on a finite-dimensional attractor of a homoeomorphism is no more complicated than the dynamics that can arise in a finite-dimensional system.

However, it is natural to conjecture that it should in fact be possible to construct a homeomorphism $f:\R^{4k+4}\to\R^{4k+4}$ such that in Theorem \ref{main} in fact $A_f=A$, i.e.\ finite-dimensional attractors of homeomorphisms can always be realised, along with their dynamics, as attractors in finite-dimensional spaces. Even if one allows the finite-dimensional map to be continuous rather than a homeomorphism, to our knowledge this problem is still open.

We can reformulate this problem in a more topological way. Suppose that $X \subset B(0,R)$ is a cellular subset of $\R^n$ written as the intersection of a decreasing sequence of cells $C_j$,
$$
X=\bigcap_{j=1}^\infty C_j,
$$
where we assume without loss of generality that $C_1 \subset B(0,R)$. Also, let a homeomorphism $f : X \to X$ be given. Proposition \ref{prop:Klee} is a controlled extension result: it provides a homeomorphism $\hat{f} : \R^{2n} \to \R^{2n}$ that extends $f$ (in the sense that $\hat{f}(x,0_n) = f(x)$ for every $x \in X$) and such that \begin{equation} \label{klee} \hat{f}(B(0,r) \times B(0,r)) \subset B(0,r) \times B(0,r)\end{equation} for every $r \geq R$. We are now going to show that if the extension $\hat{f}$ can also be controlled near $X$ (and not only for $r \geq R$), then the answer to the question above is in the affirmative.

\begin{proposition} Suppose that Proposition \ref{prop:Klee} could be strengthened so that, in addition to \eqref{klee}, the relation \begin{equation} \label{respects} \hat{f}(C_j \times B(0,r)) \subset C_j \times B(0,r)\end{equation} held for every $R \geq r \geq 0$ and every $j$. Then in Theorem \ref{main} we could achieve $A_f = A$.
\end{proposition}
\begin{proof} Let $e$, $X_1$, $B(0,R)$, $h$, $\hat{h}$ and $f_1$ be as in the proof of the theorem. Find $m \geq 1$ such that $h^m(\bar{B}(0,R)) \subset B(0,R)$, so in particular $\hat{h}^m(B) \subset B$ where $B = B(0,R) \times B(0,R)$ as before. Also, for each $j=1,2,\ldots$ let \[C_j = h^{j-1}(\bar{B}(0,R));\] clearly $C_j$ is a decreasing sequence of cells whose intersection is $X_1$. Now pick an extension $\hat{f}_1 : \R^{4k+4} \to \R^{4k+4}$ that satisfies both \eqref{klee} and \eqref{respects}. Define $f = \hat{h}^m \circ \hat{f}_1$. The same argument given in Theorem \ref{main} proves that for any bounded set $C$ there is an iterate $f^n(C) \subset B$, so $f$ has a global attractor $A_f \subset B$. Clearly $A \subseteq A_f$. Moreover, observe that \[\hat{h}^m\hat{f}_1(C_j \times B(0,R)) \subseteq \hat{h}^m(C_j \times B(0,R)) \subseteq C_{j+m} \times B(0,\nicefrac{R}{2^m})\] so in particular \[f(C_j \times B(0,R)) \subset C_{j+1} \times B(0,\nicefrac{R}{2})\] and consequently \[f^n(B) \subset f^n(C_1 \times B(0,R)) \subset C_{n+1} \times B(0,\nicefrac{R}{2^n}).\] Since $A_f = \bigcap_n f^n(B)$, we get \[A_f = \bigcap_{n=1}^{\infty} f^n(B) = \bigcap_{n=1}^{\infty} (C_{n+1} \times B(0,\nicefrac{R}{2^n})) = X_1 \times \{0\} = A,\] as required.
\end{proof}

As remarked in the introduction, the problem for semiflows seems much more difficult and is still entirely open.

  \section{References}

  \parindent0pt\parskip3pt

A.V. Babin \& M.I. Vishik (1992) {\it Attractors of evolution equations}. North-Holland Publishing Co., Amsterdam.

K. Borsuk (1970) A note of the theory of shape of compacta. {\it Fund. Math,} {\bf 67}, 265--278.

K. Borsuk (1975) {\it Theory of shape}. PWN, Warsaw.

M. Brown (1960) A proof of the generalized Schoenflies theorem. {\it Bull. Amer. Math. Soc.} {\bf 66}, 74--76.

V.V. Chepyzhov \& M.I. Vishik (2002) {\it Attractors for equations of mathematical physics}. American Mathematical Society, Providence, RI.

I.D. Chueshov (2002) {\it Introduction to the theory of infinite-dimensional dissipative systems}, University Lectures in Contemporary Mathematics, AKTA, Kharkiv.

H. Crauel (2001) Random point attractors versus random set attractors. {\it J. Lond. Math. Soc.} {\bf 63}, 413--427.

R.J. Daverman (1986) {\it Decompositions of manifolds}. Academic Press Inc., London.

A. Eden, C. Foias, B. Nicolaenko, \& R. Temam (1989) {\it Exponential attractors for dissipative evolution equations}. Wiley, New York.

A. Eden, V. Kalanarov, \& S. Zelik (2011) Counterexamples to the regularity of Mane projections and global attractors, arXiv:1108.0217.

C. Foias, G.R. Sell, \& R. Temam (1988) Inertial manifolds for nonlinear evolution equations. {\it J. Diff. Eq.} {\bf 73}, 309--353.

B.M. Garay (1991) Strong cellularity and global asymptotic stability. {\it Fund. Math.} {\bf 138}, 147-–154.
1991.

B. G\"unther (1995) Construction of differentiable flows with prescribed attractor. {\it Topology Appl.} {\bf 62}, 87-–91.

 B. G\"unther \& J. Segal (1993) Every attractor of a flow on a manifold has the shape of a finite polyhedron. {\it Proc. Amer. Math. Soc} {\bf 119}, 321-–329.

J.K. Hale (1988) {\it Asymptotic behavior of dissipative systems}. American Mathematical Society, Providence, RI.

J. Heinonen (2003) {\it Geometric embeddings of metric spaces}. Report University of Jyv\"askyl\"a Department of Mathematics and Statistics, 90.

W. Hurewicz \& H. Wallman (1941) {\it Dimension theory}. Princeton University Press, Princeton, NJ.

D.M. Hyman (1969) On decreasing sequence of compact absolute retracts. {\it Fund. Math,} {\bf 64}, 91--97.

V.L. Klee (1955) Some topological properties of convex sets. {\it Trans. Amer. Math. Soc.} {\bf 78}, 30-–45.

O.A. Ladyzhenskaya (1991) {\it Attractors for Semigroups and Evolution Equations}. Cambridge University
Press, Cambridge, England.

  D.R. McMillan (1964) A criterion for cellularity in a manifold. {\it Ann. Math.} {\bf 79}, 327--337.

  R.A. McCoy (1973) Cells and cellularity in infinite-dimensional normed linear spaces. {\it Trans. Amer. Math. Soc.} {\bf 176}, 401--410.


K. Menger (1926) \"Uber umfassendste $n$-dimensionale Mengen. {\it Proc. Akad. Wetensch. Amst.} {\bf 29}, 1125--1128.

G. N\"obeling (1931) \"Uber eine $n$-dimensionale Universalmenge im $R^{2n+1}$. {\it Math. Ann.} {\bf 104}, 71--80.

E. Pinto de Moura, J.C. Robinson, \& J.J. S\'anchez-Gabites (2011) Embedding of global attractors and their dynamics {\it Proc. Amer. Math. Soc.} {\bf 10}, 2497--3512.

J.C. Robinson (1999) Global Attractors: Topology and finite-dimensional dynamics. {\it J. Dynam.
Differential Equations} {\bf 11}, 557-–581.

J.C. Robinson (2001) {\it Infinite-dimensional dynamical systems}. Cambridge University Press, Cambridge, England.

J.C. Robinson (2011) {\it Dimensions, embeddings, and attractors}. Cambridge University Press, Cambridge, England.

A.V. Romanov (2000) Finite-dimensional limiting dynamics for dissipative parabolic equations. {\it Sb.
Math.} {\bf 191}, 415-–429.

R. Temam (1988) {\it Infinite-dimensional dynamical systems in mechanics and physics}. Springer, New York.

\end{document}